\documentclass[11pt]{amsart}

\usepackage[a4paper,hmargin=3.5cm,vmargin=4cm]{geometry}
\usepackage{amsfonts,amssymb,amscd,amstext,verbatim}

\usepackage[utf8]{inputenc}
\usepackage{hyperref}

\usepackage{fancyhdr}
\usepackage{times}

\usepackage{enumerate}
\usepackage{titlesec}
\usepackage{mathrsfs}

\titleformat{\section}
{\filcenter\bfseries\large} {\thesection{.}}{0.2cm}{}
\titleformat{\subsection}[runin]
{\bfseries} {\thesubsection{.}}{0.15cm}{}[.]
\titleformat{\subsubsection}[runin]
{\em}{\thesubsubsection{.}}{0.15cm}{}[.]

\usepackage[up,bf]{caption}

\newtheorem{theorem}{Theorem}[section]

\newtheorem{lemma}[theorem]{Lemma}
\newtheorem{corollary}[theorem]{Corollary}

\theoremstyle{definition}

\newtheorem{remark}[theorem]{Remark}

\newtheorem{example}[theorem]{Example}

\numberwithin{equation}{section}
\numberwithin{figure}{section}

\newcommand\Ccal{\mathcal{C}}

\newcommand\Ocal{\mathcal{O}}

\newcommand\Cscr{\mathscr{C}}

\newcommand\Jscr{\mathscr{J}}

\newcommand\Oscr{\mathscr{O}}

\newcommand\C{\mathbb{C}}

\newcommand\N{\mathbb{N}}

\newcommand\R{\mathbb{R}}

\newcommand\Z{\mathbb{Z}}

\newcommand\wt{\widetilde}
\newcommand\wh{\widehat}

\def\sing{\mathrm{sing}}

\begin{document}

\fancyhead[LO]{Divisors defined by noncritical functions}
\fancyhead[RE]{F.\ Forstneri\v c} 
\fancyhead[RO,LE]{\thepage}

\thispagestyle{empty}

\vspace*{1cm}
\begin{center}
{\bf\LARGE Divisors defined by noncritical functions}

\vspace*{0.5cm}

{\large\bf  Franc Forstneri\v c} 
\end{center}


\vspace*{1cm}

\begin{quote}
{\small
\noindent {\bf Abstract}\hspace*{0.1cm}
In this paper we show that every complex hypersurface $A$ in a Stein manifold $X$ 
with $H^2(X;\Z)=0$ is the divisor of a holomorphic function on $X$ 
which has no critical points in $X\setminus A_\sing$. 
A similar result is proved for complete intersections of higher codimension.

\vspace*{0.2cm}

\noindent{\bf Keywords}\hspace*{0.1cm} Stein manifold, divisor, noncritical function, complete intersection

\vspace*{0.1cm}


\noindent{\bf MSC (2010):}\hspace*{0.1cm}}  32C25, 32E10, 32E30, 32S20
\end{quote}


\section{Introduction} 
\label{sec:intro}

In this paper we obtain some new results on the classical subject of 
complete intersections in Stein manifolds. For the definitions and background on this topic 
we refer to the surveys of Forster \cite{Forster1983} and Schneider \cite{Schneider1982MA} 
(see also \cite[\S 8.5]{Forstneric2017E}).  
We begin with the following result on complex hypersurfaces 
which correspond to principal divisors. 

%
%
\begin{theorem}\label{th:main1}
Let $A$ be a closed complex hypersurface in a Stein manifold $X$.
Assume that there is a continuous function $h$ on $X$ which is holomorphic in a 
neighborhood of $A$ and whose divisor equals $A=h^{-1}(0)$.
Then there exists a holomorphic function $f$ on $X$ whose divisor equals $A$
and whose critical points are precisely the singular points of $A$:
\[
	\mathrm{Crit}(f)=A_\sing.
\]
\end{theorem}

What is new is that the defining function $f\in\Ocal(X)$ of $A$
can be chosen to have no critical points in $X\setminus A$. (In fact, if the divisor 
of $f\in \Ocal(X)$ equals $A$, then a point $x\in A$ is a critical point of $f$ if and only if $x$ a singular 
point of $A$.) The family $\{f^{-1}(c): c\in \C\}$ is then a foliation of $X$ by closed complex 
hypersurfaces all which, except perhaps the zero fibre $A=f^{-1}(0)$, are smooth. 
In particular, if the hypersurface $A$ in Theorem \ref{th:main1} is smooth,
then it is defined by a holomorphic function without any critical points on $X$. 

Without the condition that $f$ be noncritical on $X\setminus A$, the result is well known and follows 
from the isomorphisms (see e.g.\  \cite[\S 5.2]{Forstneric2017E}) 
\begin{equation}\label{eq:H1}
	H^1(X;\Ocal_X^*) \cong H^1(X;\Ccal_X^*) \cong H^2(X;\Z)
\end{equation}
showing that a second Cousin problem on a Stein manifold
is solvable by holomorphic functions if it is solvable by continuous functions.
Here, $\Ocal_X^*\subset \Ccal_X^*$ are the sheaves of nonvanishing holomorphic and 
continuous functions on $X$, respectively.

Let us consider the special case of Theorem \ref{th:main1} when $H^2(X;\Z)=0$. In view of \eqref{eq:H1} 
it follows that every divisor on $X$ is a principal divisor, and hence every complex hypersurface
in $X$ is defined by a single holomorphic equation. This gives the following corollary
which answers a question asked by Antonio Alarc{\'o}n (private communication).

\begin{corollary}\label{cor:divisors}
Let $X$ be a Stein manifold with $H^2(X;\Z)=0$. For every closed complex hypersurface
$A$  in $X$ there exists a holomorphic function $f\in\Ocal(X)$ whose divisor equals $A$ and whose critical points 
are precisely the singular points of $A$.
\end{corollary}

We compare Theorem \ref{th:main1} with another result in the literature.
Assume that $A\subset X$ is a smooth complex hypersurface satisfying the 
assumption of Theorem \ref{th:main1}. Its normal bundle 
$N_{A/X}=TX|_A/TA$ is then a trivial line bundle. By \cite[Corollary 2.10]{Forstneric2003AM}
(see also \cite[Corollary 9.16.2]{Forstneric2017E})  there exists
a holomorphic function $f$ on $X$ without critical points 
such that $A$ is a union of connected components of the zero fibre $f^{-1}(0)$.
The improvement in Theorem \ref{th:main1} and Corollary \ref{cor:divisors} is that $f$ 
can be chosen such that its zero fibre is exactly $A$.

We also have the following analogous result for submanifolds of higher codimension.

%
%
\begin{theorem} \label{th:main2}
Assume that $X$ is a Stein manifold of dimension $n>1$, $q\in\{1,\ldots,n-1\}$,  
and $h=(h_1,\ldots,h_q)\colon X\to\C^q$ is a continuous map which is a holomorphic submersion 
in a neighborhood of its zero fibre $A=h^{-1}(0)$. If there exists a $q$-tuple  of continuous 
$(1,0)$-forms $\theta=(\theta_1,\ldots,\theta_q)$ on $X$ which are pointwise linearly independent at every point of $X$
and agree with $dh=(dh_1,\ldots,dh_q)$ in a neighborhood of $A$, then there is a holomorphic submersion 
$f\colon X\to\C^q$ such that $A=f^{-1}(0)$ and $f-h$ vanishes to a given finite order on $A$.
Such $\theta$ always exists if $q\le \frac{n+1}{2}$; equivalently, if $\dim A\ge \left[\frac{n}{2}\right]$.
\end{theorem}

Note that the submersion $f\colon X\to\C^q$ in Theorem \ref{th:main2} defines a nonsingular 
holomorphic foliation on $X$ by closed embedded complete intersection submanifolds 
$\Sigma_c=\{f=c\}$ $(c\in\C^q)$ of codimension $q$ such that $\Sigma_0=A$.

Let us analyse the conditions in Theorem \ref{th:main2} more closely.
Assume that $A$ is a closed complex submanifold (not necessarily connected)
of pure codimension $q$ in a Stein manifold $X$.
Then, the normal bundle $N_{A/X}$ is trivial if and only if there is a
neighborhood $U\subset X$ of $A$ and a holomorphic submersion
$h\colon U\to \C^q$ such that $A=h^{-1}(0)$; the nontrivial (only if) direction 
follows from the tubular neighborhood theorem of Docquier and Grauert
\cite{DocquierGrauert1960} (see also \cite[Theorem 3.3.3, p.\ 74]{Forstneric2017E}).
Assume that this holds. If $h$ can be extended from a possibly smaller neighborhood of $A$
to a continuous map $\tilde h\colon X\to  \C^q$ satisfying $\tilde h^{-1}(0)=A$,
then an application of the Oka principle 
shows that $\tilde h$ can be deformed to a holomorphic map $f\colon X\to\C^q$
that agrees with $h$ to a given order on $A$ and satisfies $f^{-1}(0)=A$; 
in other words, $f$ defines $A$ as a complete intersection in $X$.
(See \cite[Theorem 8.5.6]{Forstneric2017E} which is a special
case of \cite[Theorem 8.6.1]{Forstneric2017E}. The Oka principle 
for complete intersections was first proved by  
Forster and Ramspott \cite[\S 4]{ForsterRamspott1966IM-2}, \cite{ForsterRamspott1966IM-1}.)
By topological reasons, such an extension $\tilde h$ always exists if 
$\dim A$ is sufficiently low compared to $\dim X$. 
In particular, we recall the following results in this direction.

\begin{theorem} \label{th:completeint}
{\rm (Forster and Ramspott \cite[Satz 11, Satz 12, Satz 13]{ForsterRamspott1966IM-2})}
Let $A$ be a closed complex submanifold of pure dimension in a Stein manifold $X$. 
\begin{itemize}
\item[\rm (a)] If $\dim A<\frac{1}{2}\dim X$, then $A$ is a 
complete intersection if and only if the normal bundle $N_{A/X}$ is trivial. 
\vspace{1mm}
\item[\rm (b)] 
If $X=\C^n$ then the same conclusion as in (a) holds if $\dim A\le \frac{2}{3}(n-1)$.
\vspace{1mm}
\item[\rm (c)] 
If $X=\C^n$ with $n\le 6$ then $A$ is a complete intersection if and only if $c_1(A)=0$
(i.e., the first Chern class of $A$ vanishes).
\end{itemize}
\end{theorem}

The second  assumption on $h$ in Theorem \ref{th:main2}, that
the differential $dh=(dh_1,\ldots, dh_q)$ extends to a $q$-tuple of pointwise 
linearly independent $(1,0)$-forms $(\theta_1,\ldots,\theta_q)$ on $X$
(a {\em $q$-coframe} in the terminology of \cite{Forstneric2003AM}),
holds by topological reasons if 
\begin{equation}\label{eq:Abig}
	q=n -\dim A \le \frac{n+1}{2} \ \ \Longleftrightarrow \ \ \dim A \ge \left[\frac{n}{2}\right].
\end{equation}
(See \cite[Theorem 8.3.1(c), p.\ 361]{Forstneric2017E}.)
From (\ref{eq:Abig}) and Theorem \ref{th:completeint} we obtain items (1)--(3) 
in the following corollary to Theorem \ref{th:main2}. 
For the last statement in part (1), note that every holomorphic vector
bundle over an open Riemann surface is holomorphically trivial by the 
Oka-Grauert principle (see \cite[Theorem 5.3.1(iii), p.\ 213]{Forstneric2017E}).
The submanifold $A$ in the corollary is assumed to be of pure dimension.

%
%
\begin{corollary}\label{cor:Cn}
\begin{enumerate}
\item 
If $X$ is a Stein manifold of odd dimension $n=2k+1$ and $A\subset X$ is a closed complex 
submanifold of dimension $k$ with trivial normal bundle, then 
there exists a holomorphic submersion $f\colon X\to\C^{k+1}$ such that $A=f^{-1}(0)$. 
\vspace{1mm}
\item
If $A$ is a closed complex submanifold of $\C^n$ with trivial normal bundle and 
\[
	\left[\frac{n}{2}\right] \le  \dim A\le \frac{2}{3}(n-1),
\]
then $A$ is the zero fibre of a holomorphic submersion $f\colon \C^n\to\C^{n-\dim A}$.
\vspace{1mm}
\item
If $A$ is a closed complex submanifold of $\C^n$ with 
$\left[\frac{n}{2}\right] \le  \dim A <n\le 6$, then $A$ satisfies the conclusion in (2)
if any only if $c_1(A)=0$.
\vspace{1mm}
\item
If  $A$ is an algebraic submanifold of $\C^n$ of pure codimension $2$ with topologically trivial 
canonical bundle, then $A$ is the zero fibre of a holomorphic 
submersion $f=(f_1,f_2) \colon \C^n\to\C^{2}$.
\end{enumerate}
\end{corollary}

Item (4) relies on a result of Forster and Ohsawa \cite[Corollary 3.2]{ForsterOhsawa1985},
noting also that a codimension $2$ submanifold $A\subset \C^n$ satisfies the lower
bound (\ref{eq:Abig}) when $n\ge 4$, while the case of a complex curve in $\C^3$
is covered by part (1). Forster and Ohsawa 
actually proved that the ideal of $A$ is generated by two entire functions of finite order, 
thereby answering a question of Cornalba and Griffiths \cite{CornalbaGriffiths1975}. 
On the other hand, the exist smooth algebraic curves in $\C^3$ whose ideal is not  
generated by two polynomials; see the discussion and references in  \cite{ForsterOhsawa1985}.
We do not know whether the submersion $f=(f_1,f_2)\colon \C^n\to\C^2$
in part (4) can be chosen of finite order; our proof does not give this. In fact,
the only known result in this directions seems to be the one of Ohsawa and the author
\cite{ForstnericOhsawa2013} to the effect that every compact Riemann surface 
with a puncture admits a noncritical function of finite order.

We list a few special cases of Corollary \ref{cor:Cn} in low dimensions.

\begin{example}
If $A\subset \C^n$ is a closed complex submanifold of pure dimension $k$ 
with trivial normal bundle, then $A$ is the zero fibre of a holomorphic submersion $\C^n\to\C^{n-k}$
in each of the following cases: $k=1$, $n\in \{2,3\}$; $k=2,\ n\in \{4,5\}$; $k=3,\ n\in \{6,7\}$.
\end{example}

When comparing Theorems \ref{th:main1} and \ref{th:main2}, the reader may be led to ask 
whether the latter result holds under the weaker assumption
that the map $h=(h_1,\ldots,h_1)\colon X \to\C^q$, which is holomorphic in a neighborhood of 
the zero-fibre $A=h^{-1}(0)$, generates the ideal of $A$ at every point.
In this case, the expected conclusion would be that $A=f^{-1}(0)$ for some holomorphic map
$f\colon X\to \C^q$ which is a submersion on $X\setminus A$.
However, we do not know the answer to this question, and we offer an explanation
why our proof does not apply in this case. When $q=1$, 
a generic perturbation of the function $h$ in Theorem \ref{th:main1} which is fixed to 
the second order on $A$ yields a holomorphic function without critical points in a deleted neighborhood
of $A$ in $X$ (see \cite[Lemma 2.9]{Forstneric2016JEMS}).
By adjusting the methods of the paper \cite{Forstneric2003AM} to functions avoiding the value $0\in\C$,
we then show that $h$ can be deformed to a holomorphic function $f\in \Ocal(X)$ 
that agrees with $h$ to a given finite order on $A$, has no critical points on $X\setminus A$,
and satisfies $f^{-1}(0)=A$. (See the proof of Theorem \ref{th:main1} in \S \ref{sec:proofs}.)
On the other hand, when $q>1$ we do not know whether
there always exists a holomorphic map $h\colon U\to\C^q$ in a neighborhood of $A=h^{-1}(0)$
which is a submersion on $U\setminus A$. 
Indeed, a generic choice of $h$ may have branch locus of dimension $q-1>0$ 
in $U\setminus A$. This problem was already present in the analysis in \cite{Forstneric2016JEMS} 
and prevented us from extending the result on the existence of holomorphic submersions $X\to \C^q$ from 
a Stein manifold $X$ of dimension $n$ to $\C^q$ for any $q\in\{1,\ldots, \left[\frac{n+1}{2}\right]\}$, 
obtained in \cite{Forstneric2003AM}, to Stein spaces with singularities when $q>1$. 

Clearly, a complete intersection submanifold $A\subset X$ in Theorem \ref{th:main2} 
is contained in a smooth hypersurface $H\subset X$.
Indeed, if $A=f^{-1}(0)$ for a holomorphic submersion $f=(f_1,\ldots,f_q)\colon X\to\C^q$
then the preimage $H=f^{-1}(H')$ of any smooth complex hypersurface $H'\subset \C^q$ 
with $0\in H'$ is such. The proof of Theorem \ref{th:main1} also gives the following result 
which holds in a bigger range of dimensions than Theorem \ref{th:main2}.

\begin{corollary}\label{cor:hypersurface}
Let $A$ be a complex submanifold of dimension $k$ in a Stein manifold $X$.
If  $\dim X>\left[\frac{3k}{2}\right]$ (equivalently, $3k+1\le 2\dim X$) 
then there exists a holomorphic foliation of $X$ by closed
complex hypersurfaces such that $A$ is contained in one of the leaves.
In particular, $A$ is contained in a smooth complex hypersurface $H\subset X$.
\end{corollary}

The last statement in the corollary is 
\cite[Theorem 1.3]{JelonekKucharz2016} due to Jelonek and Kucharz.
Their paper also contains results on this subject in the category of affine algebraic 
manifolds and includes references to previous works. 

The proof  of Corollary \ref{cor:hypersurface} goes as follows. 
The dimension assumption on $A$ is equivalent to
$q:=\dim X -k> \left[\frac{k}{2}\right]$. By \cite[Corollary 8.3.2(2)]{Forstneric2017E}
it follows that the conormal bundle of $A$ in $X$ (a holomorphic vector bundle of rank $q$ over $A$) 
admits a nonvanishing holomorphic section $\xi$. 
By the Docquier-Grauert theorem \cite{DocquierGrauert1960} 
(see also \cite[Theorem 3.3.3]{Forstneric2017E})
there is a holomorphic function $h$ in an open neighborhood of $A$
that vanishes on $A$ and satisfies $dh_x=\xi_x\ne 0$ for all $x\in A$. 
Theorem \ref{th:interpolation} then furnishes a function $f\in\Ocal(X)$ 
without critical points that agrees with $h$ to the second order on $A$. 
Then, $\{f^{-1}(c):c\in\C\}$ is a foliation of $X$ 
by closed complex hypersurfaces such that $A\subset f^{-1}(0)$.

%
%
%
%
\section{Proof of Theorems \ref{th:main1} and \ref{th:main2}}
\label{sec:proofs}

We shall need the following approximation result for holomorphic submersions
between Euclidean spaces whose range avoids  the origin. Without the latter condition
on the range, this is \cite[Theorem 9.12.2]{Forstneric2017E}. 
(See also \cite[Theorem 9.12.1]{Forstneric2017E} 
for the case when $q=1$ and $K$ is polynomially convex.
These results originate in \cite[\S 3.1]{Forstneric2003AM}.)

\begin{theorem}
\label{th:approximation}
Let $K$ be a compact convex set in $\C^n$, $q\in \{1,\ldots,n-1\}$, and let 
$f \colon U\to \C^q_*=\C^q\setminus\{0\}$
be a holomorphic submersion on a neighborhood $U \subset \C^n$ of $K$. 
Given  $\epsilon>0$ there exists a holomorphic submersion $g \colon \C^n\to \C^q_*$ 
such that $\sup_K |f-g| <\epsilon$.
\end{theorem}

\begin{proof}
We may assume that the set $U$ is convex. Consider first the case $q=1$. The function 
$f\colon U\to\C_*=\C\setminus\{0\}$ 
satisfies $df_z\ne 0$ at every point $z\in U$. Let $h\colon U\to \C$ be a holomorphic logarithm of $f$,
so $f=e^h$. Then, $df=e^h dh$ and hence $h$ has no critical points on $U$. 
By \cite[Theorem 9.12.1]{Forstneric2017E} we can approximate $h$ as closely as desired uniformly
on $K$ by a holomorphic function $\tilde h\colon \C^n\to\C$ without critical points on $\C^n$.
The function $g=e^{\tilde h}\colon \C^n\to \C_*$ then clearly satisfies the conclusion of the theorem.

Assume now that $q>1$. In this case, the  simple trick of taking the logarithms does not work since 
the individual components of $f$ may have zeros on $U$. Instead, we proceed as follows.
(For the details we refer to \cite[proof of Theorem 9.12.2]{Forstneric2017E} 
or \cite[\S 3.1]{Forstneric2003AM}.)

Pick a slightly bigger compact convex set $L\subset U$ with $K\subset \mathring L$.  
We begin by approximating $f$ uniformly  on $L$ by a polynomial map $P\colon\C^n\to\C^q$. 
Assuming as we may that the approximation is sufficiently close and $P$ is chosen generic, the set 
\begin{equation}\label{eq:Sigma}
	\Sigma = \{z\in \C^n: P(z)=0\ \ \text{or}\ \ \mathrm{rank}\, dP_z<q\}
\end{equation}
is an algebraic subvariety of $\C^n$ of codimension $\ge \min\{q,n-q+1\}\ge 2$
which does not intersect $K$. In particular, $\Sigma$ does not contain any hypersurfaces.
Hence there is a biholomorphic map $\phi\colon \C^n\to \phi(\C^n)\subset \C^n\setminus \Sigma$
which approximates the identity as closely as desired uniformly on $K$
(see \cite[Corollary 4.12.2]{Forstneric2017E}). The map $g=P \circ \phi \colon \C^n\to \C^q_*$ is 
then a holomorphic submersion approximating $f$ on $K$.
\end{proof}

\begin{remark}\label{rem:avoiding}
Note that for any $q>1$, Theorem \ref{th:approximation} holds with the same proof if we replace
the origin $0\in \C^q$ by an algebraic subvariety $Y\subset \C^q$ of codimension $>1$.
In this case, one replaces $\Sigma$ \eqref{eq:Sigma} by the subvariety
$
	\Sigma = \{z\in \C^n: P(z)\in Y\ \text{or}\ \mathrm{rank}\, dP_z<q\}.
$
\qed\end{remark}

With Theorem \ref{th:approximation} in hand, the proofs of Theorems \ref{th:main1} and  \ref{th:main2}
are obtained by following \cite[proof of Theorem 9.13.7]{Forstneric2017E}.
(The original reference for the latter result is \cite[Theorem 2.5]{Forstneric2003AM}.)
We explain these proofs and indicate the necessary modifications.

We begin by recalling some basic facts from analytic geometry. 
It is classical that a complex hypersurface $A$ in a complex manifold $X$
is locally at each point $x_0\in A$ the zero set of a single holomorphic function $h$
which generates the ideal $\Jscr_{A,x}$ of $A$ at every point $x\in A$ in an open neighborhood
of $x_0$. Every other holomorphic function $g$ on $X$ near $x_0$ which vanishes on $A$
is divisible by $h$, i.e. $g=uh$ for some holomorphic function $u$ in a neighborhood of $x_0$.
The function $g$ also generates $\Jscr_{A,x_0}$ if and only if $u(x_0)\ne 0$.
In particular, if the difference $g-h$ belongs to the square $\Jscr_{A,x_0}^2$ of the ideal
$\Jscr_{A,x_0}$ then $g$ is another local generator of $\Jscr_{A,x_0}$.
We say that $g$ agrees with $h$ to order $r\in\Z_+=\{0,1,2,\ldots\}$ 
on $A$ if their difference $g-h$ is a section of the sheaf $\Jscr_{A}^{r+1}$ 
on their common domain of definition. 

A compact set $K$ is a complex manifold $X$ is said to be $\Ocal(X)$-convex if
\[
	K=\wh K =\{p\in X: |f(p)| \le \sup_K |f|\ \ \forall f\in \Ocal(X)\}.
\]
In the next two lemmas we make the following assumptions:
\begin{itemize}
\item $X$ is a Stein manifold, 
\vspace{1mm}
\item $A$ is a closed complex hypersurface in $X$,
\vspace{1mm}
\item $K\subset L$ are compact $\Ocal(X)$-convex subsets of $X$ with $K\subset \mathring L$, and
\vspace{1mm}
\item $h$ is a holomorphic function in a neighborhood $U\subset X$ of $A$ 
which generates the ideal sheaf $\Jscr_A$ at every point.
\end{itemize}

\begin{lemma}\label{lem:extend1}
Let $r\in\N$. Assume that $f$ is a continuous function on $X$ with $f^{-1}(0)=A$ 
that is holomorphic on a  neighborhood $V$ of $K$, has no critical points on $V\setminus A$, 
and agrees with $h$ to order $r$ on $A\cap V$. Then there exists a continuous function
$g$ on $X$ with $g^{-1}(0)=A$ that is holomorphic 
in a neighborhood $W\subset X$ of $K\cup (A\cap L)$, has no critical points on $W\setminus A$, 
approximates $f$ as closely as desired uniformly on $K$, and agrees with $h$ to order $r$ 
along $A$. In particular, $g$ generates the ideal sheaf $\Jscr_A$ on $W$.
\end{lemma}

\begin{lemma}\label{lem:extend2}
Assume that $g$ is a continuous function on $X$ with $g^{-1}(0)=A$ that is holomorphic 
in a neighborhood $W\subset X$ of $K\cup (A\cap L)$, has no critical points on $W\setminus A$, 
and agrees with $h$ to order $r$ along $A$.
Given $r\in \N$ there exists a continuous function $\tilde f\colon X\to\C$  
with $\tilde f^{-1}(0)=A$ that is holomorphic on an open neighborhood $\wt V$ of $L$, 
agrees with $g$ to order $r$ along $A\cap\wt V$, approximates $g$ as closely as desired uniformly
on $K$, and has no critical points in $\wt V\setminus A$.
\end{lemma}

Assume these two lemmas for the moment.

%
%
\begin{proof}[Proof of Theorem \ref{th:main1}] 
We use an induction scheme as in \cite[proof of Theorem 4.1]{Forstneric2016JEMS}),
or in the proof of the Oka principle \cite[\S 5.10--\S 5.12]{Forstneric2017E}. 
Let the function $h\colon X\to \C$ be as in the theorem; in particular, $h$ is 
holomorphic on a neighborhood $U\subset X$ of its zero fibre $A=h^{-1}(0)$. Pick a  sequence
$K_1\subset K_2\subset \cdots \subset \bigcup_{j=1}^\infty K_j =X$ of compact $\Ocal(X)$-convex sets
such that $K_1\subset U$ and $K_j\subset \mathring K_{j+1}$ for all $j\in\N$. 
Fix an integer $r\in\N$. We inductively construct a sequence of continuous functions 
$f_j\colon X\to \C$, with $f_1=h$, such that the following hold for every $j\in \N$:
\begin{itemize}
\item[\rm (a$_j$)] $(f_j)^{-1}(0)=A$, $f_j$ is holomorphic on a neighborhood $U_j$ of $K_j$,
has no critical points in $U_j\setminus A$, and agrees with $h$ to order $r$ along $A\cap U_j$;
\vspace{1mm}
\item[\rm (b$_j$)] $f_{j+1}$ approximates $f_j$ as closely as desired uniformly on $K_j$.
\end{itemize}
At the $j$-th step of the induction we first apply Lemma \ref{lem:extend1} to the 
function $f=f_j$ with the pair $K=K_j$, $L=K_{j+1}$. This gives a continuous function 
$g=g_{j}\colon X\to \C$ with $g^{-1}(0)=A$ that is holomorphic on a 
neighborhood $W_{j+1}$ of $K_{j}\cup (A\cap K_{j+1})$, 
agrees with $h$ to order $r$ along $A\cap W_{j+1}$, approximates $f_j$ on $K_j$,
and has no critical points in $W_{j+1}\setminus A$. 
Hence, $g$ satisfies the assumptions of Lemma \ref{lem:extend2} 
with respect to the sets $K=K_j$, $L=K_{j+1}$ and $W=W_{j+1}$.
That lemma then furnishes the next function $f_{j+1}$ satisfying conditions
(a$_{j+1}$) and  (b$_j$). This completes the induction step. 
If the approximations are close enough at every step, then the sequence 
$f_j$ converges uniformly on compacts in $X$ to a holomorphic function
$f=\lim_{j\to\infty}f_j\in \Oscr(X)$ satisfying Theorem \ref{th:main1}.
\end{proof}

\begin{proof}[Proof of Lemma \ref{lem:extend1}] 
Since the function $f$ vanishes on $A$ and agrees with $h$ to order $r$ along $A\cap V$, 
we can apply \cite[Theorem 3.4.1 and Remark 3.4.4]{Forstneric2017E}  to
find a function $g$ that is holomorphic on a neighborhood  $W$ of $K\cup (A\cap L)$, 
approximates $f$ uniformly on a neighborhood of $K$ as closely as desired, 
and agrees with $h$ to order $r$ on $A\cap W$. 
(The case of the cited result for functions, which we use here, is 
an elementary consequence of Cartan's division theorem and the Oka-Weil approximation theorem.) 
Assuming as we may that $g$ is close enough to $f$ on a neighborhood of $K$, 
\cite[Lemma  2.7]{Forstneric2016JEMS} shows that $g$ has no critical points there, 
except at the singular points of $A$. Note also that $g$ has no zeros in $W\setminus A$ provided 
that the neighborhood $W$  of $K\cup (A\cap L)$ is chosen small enough.
By  \cite[Lemma  2.9]{Forstneric2016JEMS},
a generic $g$ as above has no critical points on $W\setminus A$.
Using a smooth cutoff function we can extend $g$ to a continuous
function on $X$ such that $g^{-1}(0)= A$. 
\end{proof}

\begin{proof}[Proof of Lemma \ref{lem:extend2}]
We use the same geometric scheme as in proof of \cite[Proposition 5.12.1]{Forstneric2017E}
(see also \cite[Fig.\ 5.4, p.\ 250]{Forstneric2017E}).
The proof amounts to a finite induction where at every step we enlarge the domain on which 
the function is holomorphic and noncritical off the subvariety $A$. 
We begin with the initial function $f_0=g$ on an initial compact strongly pseudoconvex
domain $W_0$ with $K\cup (A\cap L) \subset\mathring W_0\subset W_0\subset W$,
and the process terminates in finitely many steps by reaching a function $\tilde f$ which is holomorphic on a neighborhood of $L$
and satisfies the stated conditions. Every step amounts to one of the following two types of operations:

\begin{itemize}
\item[\rm (a)] {\em The noncritical case:} 
we attach a small convex bump to a given strongly pseudoconvex domain in $X$.
\vspace{1mm}
\item[\rm (b)] {\em The critical case:} we attach a handle of index $\le n=\dim X$ 
to a given strongly pseudoconvex domain in $X$. 
\end{itemize}
The details needed to accomplish these steps follow the pattern in \cite{Forstneric2003AM}. 
We explain the induction step in each of these two cases.

In case (a) we are given a pair of compact sets $D_0,D_1\subset X$ with the 
following properties:
\begin{itemize}
\item  $D=D_0\cup D_1$ is a Stein compact (i.e., it admits a basis of Stein neighborhoods),
\vspace{1mm}
\item  $\overline {D_0\setminus D_1} \cap \overline {D_1\setminus D_0}=\varnothing$, and
\vspace{1mm}
\item  the set $D_1\subset X\setminus A$ is contained in a holomorphic coordinate chart 
$V \subset X$ such that $C=D_0\cap D_1$ is convex in that chart.
(Precisely, $V$ is biholomorphic to an open subset of $\C^n$ and $C$ 
corresponds to a compact convex subset of $\C^n$.)
\end{itemize}
Furthermore,  we are given a holomorphic function $f$ on a neighborhood of $D_0$ 
without zeros or critical points on $D_0\setminus A$. 
(This is one of the functions  in the inductive construction.) By Theorem \ref{th:approximation}
we can approximate $f$ uniformly on a neighborhood of $C$ by a 
holomorphic function $\xi$ on a neighborhood of $D_1$ which has no zeros or critical points.
(Note that  $D_1\subset X\setminus A$.)
If the approximation is close enough, then by \cite[Lemma 9.12.6]{Forstneric2017E} 
we find a neighborhood $U$ of $C$ and a biholomorphic map $U\to \gamma(U)\subset X$
close to the identity such that $f=\xi \circ \gamma$ holds on $U$.
Assuming as we may that the approximations are close enough, 
\cite[Theorem 9.7.1, p.\ 432]{Forstneric2017E} furnishes 
biholomorphic maps $\alpha$ and $\beta$ on a neighborhood 
of $D_0$ and $D_1$, respectively, close to the identity on their respective domains,
such that $\alpha$ is tangent to the identity to order $r$ on $A\cap D_0$ and 
\[
	\gamma\circ \alpha =\beta\ \ \text{holds on a neighborhood of $C$}. 
\]
It follows that $f\circ\alpha = \xi \circ\beta$ on a neighborhood of $C$.
This defines a holomorphic function $\tilde f$ on a neighborhood of
$D=D_0\cup D_1$ which is close to $f$ on $D_0$ and is tangent to $f$ to order $r$
along $A\cap D_0$. Furthermore, the construction ensures that $\tilde f$ has no zeros 
or critical points on $D\setminus A$. This completes the description of the induction step in case (a).

In case (b) (which serves to change the topology of the domain),
we are given a compact strongly pseudoconvex domain $D_0\subset X$ to which 
we attach a totally real embedded disc $M\subset X\setminus (A\cup \mathring D_0)$ 
(an image of the closed ball in $\R^k$ for some $k\in \{1,\ldots,\dim X\}$)
whose boundary sphere $bM\subset bD_0\setminus A$ is  complex tangential to $bD_0$.
Furthermore, $D_0\cup M$ is a Stein compact admitting 
small strongly pseudoconvex  neighborhoods $D$ (handlebodies) 
that can be  deformation retracted onto the core $D_0\cup M$. 
(See \cite[\S 3.9]{Forstneric2017E} or \cite[\S 7.9]{CieliebakEliashberg2012} 
for a precise description of this type of configuration.)
As before, we are given a continuous function $f\colon X\to\C$ with $f^{-1}(0)=A$
which is holomorphic on a neighborhood of $D_0$ and has no critical points off $A$. 
By a $\Cscr^0$-small deformation of $f$ in a neighborhood of $M$, keeping it fixed
in a neighborhood of  $D_0$ where it is holomorphic, we may assume that $f$ is smooth and 
nonvanishing in a neighborhood of $M$ and its differential $df_x$ is $\C$-linear 
and nonvanishing at every point of $M$. 
Applying the Mergelyan approximation theorem (see \cite[3.8.1, p.\ 88]{Forstneric2017E})
we obtain a holomorphic function $\tilde f$ on a neighborhood of $D_0\cup M$ with the
desired properties. By using a smooth cutoff function we can glue $\tilde f$ 
with the original function $f$ outside a smaller neighborhood of $D_0\cup M$ in order to 
obtained a globally defined continuous function on $X$ that vanishes precisely on $A$,
is holomorphic on a compact strongly pseudoconvex neighborhood $D$ of $D_0\cup M$, 
and has no critical points on $D\setminus A$.  This completes the induction step in case (b). 
\end{proof}

The proof of Theorem \ref{th:main1} also gives the following result which 
we state for future reference. The same result  holds if $X$ is a Stein space 
and the subvariety $A\subset X$ contains the singular locus $X_{\sing}$ 
(see \cite[Theorem 4.1]{Forstneric2016JEMS}).

\begin{theorem}\label{th:interpolation}
Assume that $A$ is a closed complex subvariety in a Stein manifold $X$ and 
$h$ is a holomorphic function in a neighborhood of $A$. 
Given $r\in\N$ there exists a holomorphic function $f$ on $X$
which has no critical points on $X\setminus A$ and agrees with $h$ to order $r$ on $A$.
In particular, if $dh_x\ne 0$ for all $x\in A$, then $f$ has no critical points on $X$. 
\end{theorem}

\begin{proof}[Proof of Theorem \ref{th:main2}]
In this case, $A$ is a complex submanifold of $X$ of codimension $q$ and we are dealing with 
maps $f=(f_1,\ldots, f_q)\colon X\to\C^q$ that are submersions near $A$; the latter condition 
is stable under small perturbations which are fixed on $A$ to the second order. 
The proof follows the same scheme as that of Theorem \ref{th:main1}; see 
\cite[proof of Theorem 2.5]{Forstneric2003AM} for the details. The only difference
is that we must ensure in addition that the map $f$ 
has no zeros on $X\setminus A$. This is taken into account when dealing 
with cases (a) (the noncritical case) and (b) (the critical case) in the proof of Theorem \ref{th:main1}.

An inspection shows that step (a) goes through exactly as before.
However, in step (b) we must make an additional effort
to see that the map  $f=(f_1,\ldots,f_q)$ can be extended smoothly across the disc $M$ 
such that it has no zeros there and its differential is $\C$-linear and of maximal rank $q$ 
at every point of $M$. The first condition holds if we  keep the deformation 
uniformly sufficiently small. The second condition concerning the differential holds generically
when $q\le \frac{n+1}{2}$, so in this case there always exists a 
$q$-coframe $(\theta_1,\ldots,\theta_q)$ on $X$ extending $dh=(dh_1,\ldots, dh_q)$.
On the other hand, when $q> \frac{n+1}{2}$ and we already have a $q$-coframe 
as in  the theorem, we can achieve the required condition on the differential 
$df_x$ for points $x\in M$ in the attached disc 
by applying Gromov's h-principle \cite{Gromov1986E,Gromov1973IAN};
as shown in \cite[proof of Theorem 2.5]{Forstneric2003AM},
the partial differential relation controlling this problem is ample in the coordinate directions. 
This requires a deformation that is big in the $\Cscr^1$ norm but arbitrarily 
small in the $\Cscr^0$ norm; hence we do not introduce any zeros on $M$. 
We complete the proof as before by applying the Mergelyan theorem.
Further details are available in \cite[proof of Theorem 2.5]{Forstneric2003AM}
and also in \cite[\S 9.13]{Forstneric2017E}.
\end{proof}

The construction  in the proofs of Theorems \ref{th:main1} and  \ref{th:main2} 
is easily augmented to provide a homotopy of continuous maps $f_t\colon X\to \C^q_*$ $(t\in[0,1])$ 
such that $f_0=h$, for every $t\in [0,1]$ the map $f_t$ is holomorphic in a neighborhood $V\subset X$ of 
$A$ (independent of $t$) and agrees with $h$ to order $r$ on $A=f_t^{-1}(0)$, and 
the holomorphic map $f=f_1\colon X\to\C^q$ has maximal rank $q$ on $X$ 
(resp.\ on $X\setminus A$ if $q=1$ and we are proving Theorem \ref{th:main1}). 
For the details in a very similar situation we refer the reader to 
\cite[\S 5.13]{Forstneric2017E}.


\subsection*{Acknowledgements}
The author is supported by the research program P1-0291 and grant J1-7256 from 
ARRS, Republic of Slovenia. I wish to thank Antonio Alarc{\'o}n for having asked
the question answered by Corollary \ref{cor:divisors}, thereby inducing me
to write this note, and Frank Kutzschebauch for pointing out the connection
to the problem in Corollary \ref{cor:hypersurface}.



\bibliographystyle{amsplain}


\vspace*{0.5cm}
\noindent Franc Forstneri\v c

\noindent Faculty of Mathematics and Physics, University of Ljubljana, Jadranska 19, SI--1000 Ljubljana, Slovenia

\noindent Institute of Mathematics, Physics and Mechanics, Jadranska 19, SI--1000 Ljubljana, Slovenia

\noindent e-mail: {\tt franc.forstneric@fmf.uni-lj.si}

\end{document}